\documentclass[a4paper,11pt]{article}

\usepackage[hmargin=1in,vmargin=1in]{geometry}

\usepackage{amsmath}              
\usepackage{amssymb}              
\usepackage{amsthm}               
\usepackage{enumerate}            
\usepackage{graphicx}             

\newtheorem{theorem}{Theorem}
\newtheorem{lemma}[theorem]{Lemma}
\newtheorem{corollary}[theorem]{Corollary}
\newtheorem{conjecture}{Conjecture} 

\theoremstyle{definition}
\newtheorem*{defn}{Definition}

\newcommand{\baraw}{\alpha}
\newcommand{\bbraw}{\beta}
\newcommand{\bcraw}{\gamma}
\newcommand{\ba}{\baraw\phantom{{}+\bbraw+\bcraw}}
\newcommand{\bb}{\phantom{\baraw+{}}\bbraw\phantom{{}+\bcraw}}
\newcommand{\bc}{\phantom{\baraw+\bbraw+{}}\bcraw}
\newcommand{\bab}{\baraw+\bbraw\phantom{{}+\bcraw}}
\newcommand{\bac}{\baraw\phantom{{}+\bbraw}+\bcraw}
\newcommand{\bbc}{\phantom{\baraw+{}}\bbraw+\bcraw}
\newcommand{\babc}{\baraw+\bbraw+\bcraw}
\newcommand{\block}{\mathcal{B}}
\newcommand{\oin}{\in}
\newcommand{\path}{\mathcal{X}}
\newcommand{\R}{\mathbb{R}}
\newcommand{\regina}{\emph{Regina}}
\newcommand{\s}[1]{\mathbf{#1}}
\newcommand{\scount}{\sigma}
\newcommand{\smeanraw}{\overline{\scount}}
\newcommand{\smean}[1]{\smeanraw_{#1}}

\newcommand{\tri}{\mathcal{T}}

\begin{document}

\title{The complexity of the normal surface solution space\footnote{%
    A more detailed version of this paper will be available as
    \emph{Extreme cases in normal surface enumeration} (in preparation).}}
\author{Benjamin A.~Burton}
\date{March 11, 2010}

\maketitle

\begin{abstract}
Normal surface theory is a central tool in algorithmic three-dimensional
topology, and the enumeration of vertex normal surfaces is the
computational bottleneck in many important algorithms.
However, it is not well understood how the number of such surfaces
grows in relation to the size of the underlying triangulation.
Here we address this problem in both theory and practice.
In theory, we tighten the exponential upper bound substantially;
furthermore, we construct pathological triangulations that prove
an exponential bound to be unavoidable.
In practice, we undertake a comprehensive analysis of millions of
triangulations and find that in general the number of vertex normal
surfaces is remarkably small, with strong evidence that our pathological
triangulations may in fact be the worst case scenarios.
This analysis is the first of its kind, and the striking behaviour
that we observe has important implications for the
feasibility of topological algorithms in three dimensions.
\end{abstract}

%
%

\section{Introduction}

Geometric topology is an inherently algorithmic subject, with
fundamental questions such as the \emph{homeomorphism problem}
(find an algorithm to determine whether two given spaces
are topologically equivalent) and the \emph{identification problem}
(find an algorithm to determine the topological name and/or structure
of a given space).  Three-dimensional topology is of particular
interest, since in lower dimensions such problems become trivial
\cite{massey91}, and in higher dimensions they become unsolvable
\cite{markov60-insolubility}.

Throughout this paper we restrict our attention to \emph{closed
3-manifolds}.  In essence, a closed $3$-manifold is a compact
$3$-dimensional topological space that locally looks like $\R^3$ at
every point.
Much recent progress has been made on algorithms in 3-manifold topology.
For example:
\begin{itemize}
    \item Rubinstein gave an algorithm in 1992 for recognising
    the simplest of all closed 3-manifolds, namely the
    3-sphere \cite{rubinstein95-3sphere,rubinstein97-3sphere};
    this algorithm has been refined several times since
    \cite{burton09-quadoct,jaco03-0-efficiency,thompson94-thinposition}.

    \item In 1995, Jaco and Tollefson gave an algorithm for breaking a
    3-manifold down into a connected sum decomposition (essentially a
    topological ``prime decomposition'')
    \cite{jaco95-algorithms-decomposition}.

    \item Perelman's proof of the geometrisation conjecture in 2002
    finally resolved the general homeomorphism problem for 3-manifolds,
    completing a programme initiated decades earlier by pioneers
    such as Haken \cite{haken62-homeomorphism} and
    Thurston \cite{thurston82-geometrisation}.
    The full homeomorphism algorithm is a fusion of diverse
    and complex components, including both the 3-sphere recognition and
    connected sum decomposition algorithms above.
\end{itemize}

A recurring theme in these algorithms (and many others) is that they rely upon
\emph{normal surface theory}, a tool that allows us to convert difficult
topology problems into simpler linear programming problems.
In particular, we can search for an interesting surface within
a 3-manifold by (i)~constructing a high-dimensional polytope,
(ii)~enumerating the ``admissible'' vertices of this polytope,
and then (iii)~testing each admissible vertex to see whether it
encodes the interesting surface that we are searching for.\footnote{%
Some other algorithms (such as knot genus \cite{hass99-knotnp} and
Heegaard genus \cite{li10-genus})
replace step~(ii) with the more difficult enumeration of
a Hilbert basis for a polyhedral cone, yielding what are
known as \emph{fundamental surfaces}.}

The concept of an ``interesting surface'' depends on the application at
hand.  For instance, in the connected sum decomposition algorithm we
search for embedded spheres within our 3-manifold;
in other algorithms we might search for non-trivial embedded discs
\cite{haken61-knot} or embedded incompressible surfaces \cite{jaco84-haken}.
However, in all of these applications the high-dimensional
polytope and its admissible vertices remain the same.  That is,
the polytope vertex enumeration problem is a \emph{common component}
for all of these topological algorithms and many others besides.

Furthermore, this common vertex enumeration problem is in fact the
computational bottleneck for many of these algorithms
\cite{burton09-quadoct,burton09-ws}.
It is therefore important to improve the efficiency and
understand the complexity of this vertex enumeration problem,
since any improvements or results will have a widespread impact on
computational 3-manifold
topology as a whole.  This impact also extends beyond three
dimensions---for instance, in \emph{4-manifold topology}, to
understand whether a given triangulation represents a 4-manifold
we require all of the complex machinery of 3-sphere recognition as
discussed above.

In general, polytope vertex enumeration is difficult.
The general problem is known to be NP-hard
\cite{dyer83-complexity,khachiyan08-hard},
and the range of available algorithms is matched by a range of
pathological cases that exploit their weaknesses \cite{avis97-howgood-compgeom}.
However, in our context we have two advantages:
\begin{itemize}
    \item We are not dealing with an arbitrary polytope, but rather one
    that derives from the machinery of normal surface theory; this polytope
    is known as the \emph{projective solution space}.  Such polytopes
    have additional constraints on their dimensions and the equalities
    and inequalities that define them.

    \item We do not need to enumerate all vertices of the polytope, but
    only the \emph{admissible vertices}.  These are the vertices
    that satisfy an additional family of non-linear constraints,
    known as the \emph{quadrilateral constraints}.
\end{itemize}

These contextual advantages can be exploited in vertex enumeration
algorithms with great success; see
\cite{burton09-convert,burton09-quadoct,burton10-dd,tollefson98-quadspace}
for details.  Nevertheless, the enumeration problem remains a difficult
one.  In particular, Agol et~al.\ \cite{agol02-knotgenus} show that
determining knot genus---yet another problem that employs normal
surface theory---is in fact NP-complete.

In this paper we concern ourselves with the \emph{complexity} of the
enumeration problem.  More specifically, we focus on the
\emph{number of admissible vertices} of the projective solution space,
which we denote by $\scount$.  This quantity is important for the
following reasons:
\begin{itemize}
    \item The admissible vertex count $\scount$ gives a lower bound for
    the time complexity of vertex enumeration.  Moreover, for the
    quadrilateral-to-standard conversion algorithm (a key component of
    the current state-of-the-art enumeration algorithm), there is
    strong evidence to suggest that the running time is in fact a
    low-degree polynomial in $\scount$ \cite{burton09-convert}.

    \item Each admissible vertex corresponds to a surface in our
    3-manifold upon which we must run some subsequent test.  For some
    problems (such as Hakenness testing \cite{burton09-ws,jaco84-haken})
    this test is extremely expensive, and so the number of admissible
    vertices becomes a critical factor in the overall time complexity.
\end{itemize}

The input for a typical normal surface algorithm is a \emph{3-manifold
triangulation}, formed from $n$ tetrahedra by joining their $4n$ faces
together in pairs.  We call $n$ the \emph{size} of the triangulation;
not only does $n$ represent the complexity of the input, but both the
dimension and the number of facets of the projective solution space are
linear in $n$.

The growth of $\scount$ as a function of $n$ is currently
not well understood.  The only general theoretical bound in the literature
is $\scount \leq 128^n$, proven by Hass et~al.\ \cite{hass99-knotnp};
in the special case of a one-vertex triangulation
this has been improved to $\scount \oin O(15^n)$ \cite{burton10-dd}.
Very little is known about the growth of $\scount$ in practice, though
initial observations suggest that $\scount$ is in fact far smaller
\cite{burton09-convert}.  For example, in the proof that the
Weber-Seifert dodecahedral space is non-Haken (one of the first
significant computer proofs to employ normal surface theory),
a ``typical'' triangulation of size $n=23$ is found to generate just
$\scount=1751$ admissible vertices \cite{burton09-ws}.

In this paper we shed more light on the growth of $\scount$, including
new theoretical bounds and comprehensive practical experimentation.
Following a brief outline of normal surface theory in
Section~\ref{s-prelim}, we present the following results:
\begin{itemize}
    \item In Section~\ref{s-theory} we show that $\scount \oin O(\phi^{7n})$,
    where $\phi$ is the golden ratio
    $(1+\sqrt{5})/2$.  This tightens the general theoretical
    bound on $\scount$ from $128^n$ to just over $O(29^n)$.  We prove this
    by extending McMullen's upper bound theorem
    \cite{mcmullen70-ubt} to show that any convex polytope with
    $k$ facets must have $O(\phi^k)$ vertices.

    We push this bound from the other direction in Section~\ref{s-extreme} by
    constructing an infinite family of 3-manifold triangulations
    for which $\scount = 17^{n/4} + n/4$.  This yields the first known
    family for which $\scount$ is exponential in $n$, and disproves an
    earlier conjecture of the author that $\scount \oin O(2^n)$.
    By extending this family to all $n > 5$ we show that any theoretical upper
    bound must grow at least as fast as
    $\Omega(17^{n/4}) \simeq \Omega(2.03^n)$.

    \item In Section~\ref{s-practice} we build a comprehensive census of
    \emph{all} 3-manifold triangulations of size $n \leq 9$, and measure
    $\scount$ for each of the $\sim 150$~million triangulations that ensue.
    We find a remarkably slow growth rate---for $n > 5$ the
    worst cases are precisely the infinite family above, suggesting that
    the lower limit of $\Omega(17^{n/4}) \simeq \Omega(2.03^n)$
    may in fact be tight.
    In the average case the mean $\smeanraw$ appears to grow even slower,
    with an apparent growth rate of less than $\phi^n$ and a final mean of
    just $\smeanraw \simeq 78.49$ for $n = 9$.

    This analysis is the first of its kind, primarily because the complex
    algorithms and software required for such a comprehensive study did
    not exist until very recently \cite{burton07-nor10,burton09-convert}.
    Previous censuses have focused on restricted classes of
    triangulations (such as minimal triangulations of irreducible
    or hyperbolic manifolds
    \cite{burton07-nor10,callahan99-cuspedcensus,martelli01-or9,matveev05-or11}),
    and previous measurements of $\scount$ have been for isolated or ad-hoc
    collections of cases \cite{burton09-convert,burton09-ws,matsumoto00-fig8}.
\end{itemize}

Throughout this paper we work with Haken's original formulation of
normal surface theory \cite{haken61-knot,haken62-homeomorphism}.
Tollefson defines an alternative formulation called \emph{quadrilateral
coordinates} \cite{tollefson98-quadspace},
which is only applicable for some problems but where the
polytope becomes much simpler.  In quadrilateral coordinates an upper
bound of $\scount \leq 4^n$ can be obtained through an
analysis of \emph{zero sets} \cite{burton10-dd}, but again the growth
rate is found to be significantly slower in practice.  We address
quadrilateral coordinates in detail in the full version of this paper.

%
%

\section{Preliminaries} \label{s-prelim}

Throughout this paper we assume that we are working with a
\emph{3-manifold triangulation of size $n$}.  By this we mean a collection
of $n$ tetrahedra, some of whose $4n$ faces are affinely identified (or
``glued together'') in pairs so that the resulting topological space
is a 3-manifold (possibly with boundary).  If all $4n$ faces are
identified in $2n$ pairs then we obtain a closed 3-manifold; otherwise
we obtain a \emph{triangulation with boundary}, and the unidentified
faces become \emph{boundary faces}.  Unless otherwise specified, all
triangulations in this paper are of closed 3-manifolds.

There is no need for a 3-manifold triangulation to be rigidly
embedded in some larger space---tetrahedra can be ``bent'' or
``stretched''.  Moreover, we allow multiple vertices of the
same tetrahedron to be identified as a result of our face gluings,
and likewise with edges.  This allows us to build
triangulations using very few tetrahedra, which becomes useful for computation.

\begin{figure}[htb]
    \centering
    \includegraphics[scale=0.5]{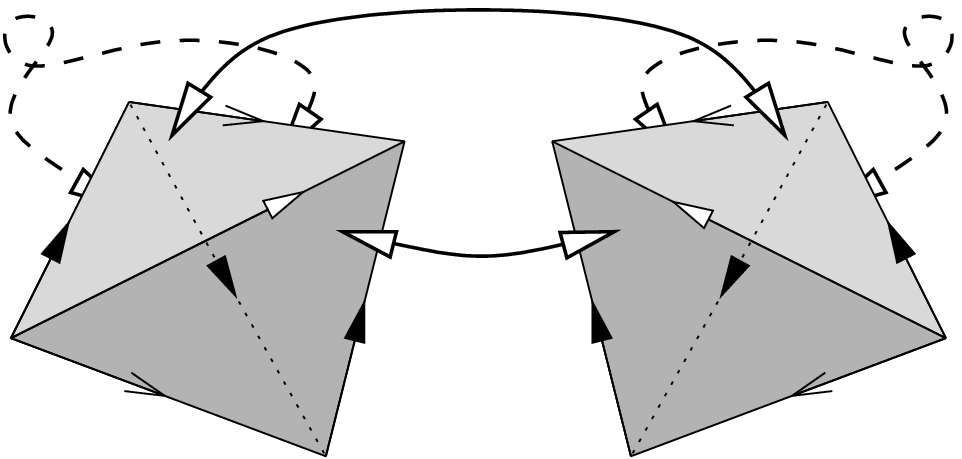}
    \hspace{2cm}
    \includegraphics[scale=0.5]{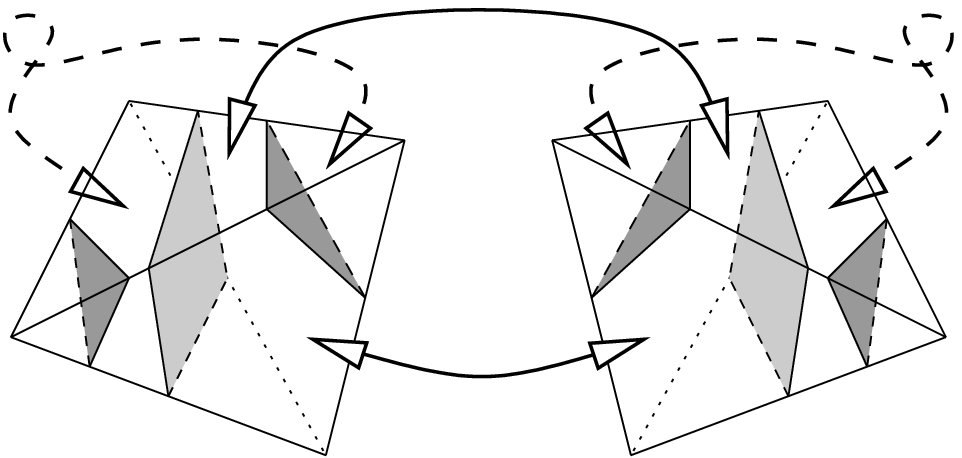}
    \caption{A 3-manifold triangulation and an embedded normal surface}
    \label{fig-s2xs1}
\end{figure}

To illustrate, the left-hand diagram of Figure~\ref{fig-s2xs1} shows a
triangulation of the product space $S^2 \times S^1$ using just
$n=2$ tetrahedra---the back two faces of each tetrahedron are identified
with a twist, and the front two faces of the left tetrahedron are
identified directly with the front two faces of the right tetrahedron.
All eight vertices become identified together, and the 12 edges become
identified in three distinct classes (represented in the diagram by three
different types of arrowhead).  We say that the resulting triangulation has
\emph{one vertex} and \emph{three edges}.

Normal surfaces were introduced by Kneser \cite{kneser29-normal}, and
further developed by Haken \cite{haken61-knot,haken62-homeomorphism}
for use in algorithms.  A \emph{normal surface} is a
2-dimensional surface embedded within a 3-manifold triangulation
that meets each tetrahedron in a (possibly empty) collection of
\emph{triangles} and/or \emph{quadrilaterals}, as illustrated
in Figure~\ref{fig-normaldiscs}.  For example, a normal surface within
our $S^2 \times S^1$ triangulation is shown on the right-hand side of
Figure~\ref{fig-s2xs1}; as a consequence of the tetrahedron gluings,
the six triangles and quadrilaterals join together to form a
2-dimensional sphere.

\begin{figure}[htb]
    \centering
    \includegraphics[scale=0.5]{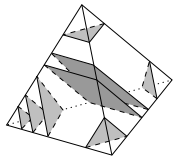}
    \caption{Normal triangles and quadrilaterals within a tetrahedron}
    \label{fig-normaldiscs}
\end{figure}

There are four distinct \emph{types} of triangle and three distinct
\emph{types} of quadrilateral within each tetrahedron (defined by which edges
of the tetrahedron they meet).  The \emph{vector representation} of a
normal surface is a collection of $7n$ integers counting the number of
pieces of each type in each tetrahedron; from this vector in $\R^{7n}$
we can completely reconstruct the original surface.  We treat surfaces
and their vectors interchangeably (so, for instance,
``adding'' two surfaces means adding their two vectors and
reconstructing a new surface from the result).

An early result of Haken is a set of necessary and sufficient conditions
for a vector to represent a normal surface: (i)~all coordinates must be
non-negative; (ii)~the vector must satisfy a set of linear homogeneous
equations (the \emph{matching equations}); and (iii)~there can be at most
one non-zero quadrilateral coordinate corresponding to each tetrahedron (the
\emph{quadrilateral constraints}).  Vectors that satisfy all of these
conditions are called \emph{admissible}.

Jaco and Oertel \cite{jaco84-haken} define the \emph{projective solution
space} to be the polytope in $\R^{7n}$
obtained as a cross-section of the cone defined by (i) and (ii) above.
A \emph{vertex normal surface} lies on an extremal ray of
this cone and is not a multiple of some smaller surface.
The vertex normal surfaces are in bijection with the admissible vertices of
the projective solution space; we let $\scount$ denote
the number of vertex normal surfaces, and we call $\scount$ the
\emph{admissible vertex count}.

The enumeration of vertex normal surfaces is a critical component---and
often the computational bottleneck---of many important topological
algorithms.  This is because one can often prove that,
if an interesting surface exists (such as an incompressible surface or
an essential sphere), then one must appear as a vertex normal surface.
See Hass et~al.\ \cite{hass99-knotnp}
for a more detailed introduction to normal surface theory and its
role in computational topology.

%
%

\section{Theoretical Bounds} \label{s-theory}

As noted in the introduction, the best bound known to date
for the admissible vertex count is $\scount \leq 128^n$, proven by
Hass et~al.\ \cite{hass99-knotnp}.  We begin by tightening this
exponential bound as follows:

\begin{theorem} \label{t-ubound}
    Let $\phi = (1+\sqrt{5})/2$.  Then the admissible vertex count
    $\scount$ is bounded above by $O(\phi^{7n}) \simeq O(29.03^n)$.
\end{theorem}

We prove this through a simple extension of McMullen's upper bound
theorem \cite{mcmullen70-ubt}.  McMullen gives a tight bound on the
number of vertices for a convex polytope with $k$ facets and $d$
dimensions; we extend this here to a loose bound that covers all
possible dimensions.

\begin{lemma} \label{l-ubound-fib}
    Let $F_0=0$, $F_1=1$, $F_2=1$, \ldots\ represent the Fibonacci
    sequence, where $F_{i+2} = F_{i+1} + F_i$.  Then for any $k \geq 3$,
    a convex polytope with precisely $k$ facets has $\leq F_{k+1}$ vertices.
\end{lemma}

\begin{proof}
    Suppose the polytope $P$ is $d$-dimensional with precisely $k$ facets.
    Then McMullen's theorem (taken in dual form) shows that $P$ has at most
    \begin{equation} \label{eqn-ubt}
    \binom{k - \lfloor \frac{d+1}{2} \rfloor}{k - d} +
    \binom{k - \lfloor \frac{d+2}{2} \rfloor}{k - d}
    \quad = \quad
    \binom{k - \lfloor\frac{d+1}{2}\rfloor}{d - \lfloor\frac{d+1}{2}\rfloor} +
    \binom{k - \lfloor\frac{d+2}{2}\rfloor}{d - \lfloor\frac{d+2}{2}\rfloor}
    \end{equation}
    vertices.\footnote{This is the number of facets of the cyclic
    $d$-dimensional polytope with $k$ vertices \cite{grunbaum03}.}
    For even $d$ this can be rewritten as
    $\binom{k-a}{a} + \binom{(k-2)-b}{b}$ for suitable integers $a,b$,
    and for odd $d$ it can be rewritten as $2\binom{(k-1)-a}{a}$
    for a suitable integer $a$.

    We now claim that $\binom{k-a}{a} \leq F_k$ for any $k,a$ with
    $k \geq 1$.  This is easily established for $k=1,2$, and the full
    claim follows from the inductive step
    $\binom{k-a}{a} = \binom{k-1-a}{a} + \binom{k-1-a}{a-1}
    = \binom{(k-1)-a}{a} + \binom{(k-2)-(a-1)}{a-1}
    \leq F_{k-1} + F_{k-2} = F_k$.

    From here our lemma is straightforward.  If $d$ is even then
    the number of vertices of $P$ is at most
    $\binom{k-a}{a} + \binom{(k-2)-b}{b} \leq F_k + F_{k-2}
    \leq F_k + F_{k-1} = F_{k+1}$,
    and if $d$ is odd then the number of vertices is at most
    $2\binom{(k-1)-a}{a} \leq 2 F_{k-1} \leq F_k + F_{k-1} = F_{k+1}$.
\end{proof}

Unlike McMullen's result, Lemma~\ref{l-ubound-fib} is not tight.
Nevertheless, it gives us a very good\footnote{Experimentation shows
that this asymptotic upper bound of $\phi^k \simeq 1.618^k$ is close
to optimal.  If we maximise equation~(\ref{eqn-ubt})
over all $d$ for each $k=100,\ldots,200$, the maximum grows at a rate
of approximately $1.613^k$.}
asymptotic upper bound of $O(\phi^k)$,
which is enough to prove our main theorem.

\begin{proof}[Proof of Theorem~\ref{t-ubound}]
    The facets of the projective solution space in $\R^{7n}$
    are defined by the $7n$ inequalities $x_1 \geq 0$, \ldots, $x_{7n} \geq 0$,
    and so there are at most $7n$ facets in total.  Lemma~\ref{l-ubound-fib}
    then shows that the projective solution space has at most
    $F_{7n+1}$ vertices, and so $\scount \leq F_{7n+1}$.
    Using the standard formula $F_k = \lfloor\phi^k/\sqrt{5}+\frac12\rfloor$
    it follows that $\scount \oin O(\phi^{7n})$.
\end{proof}

It is interesting to note that Theorem~\ref{t-ubound} makes no use of
admissibility---this suggests that, although the bound of $\phi^{7n}$
is a strong improvement on $128^n$, this bound is still very loose.
We confirm this through experimentation in Section~\ref{s-practice}.
Although we only consider closed 3-manifolds in this paper, it
should be noted that Theorem~\ref{t-ubound} and its proof apply equally
well to triangulations with boundary, and also to the
\emph{ideal triangulations} of Thurston \cite{thurston78-lectures}.

%
%

\section{Extreme Cases} \label{s-extreme}

Having tightened the upper bound from above, we now turn our attention
to limiting the upper bound from below.  We do this by building
pathological triangulations for which
$\scount \oin \Theta(17^{n/4}) \simeq \Theta(2.03^n)$.
This growth rate shows that an exponential upper bound on $\scount$ is
unavoidable, and furthermore disproves
an earlier conjecture of the author that $\scount \oin O(2^n)$.

We begin by describing 4-blocks, which are small building blocks that appear
repeatedly throughout our triangulations.  Using these building blocks,
we then construct the family of pathological triangulations
$\path_1,\path_2,\ldots$.

\begin{defn}[4-block]
    A \emph{4-block} is a triangulation with boundary,
    built from the four tetrahedra $\Delta_1,\Delta_2,\Delta_3,\Delta_4$
    using the following construction.

    \begin{figure}[htb]
        \centering
        \includegraphics{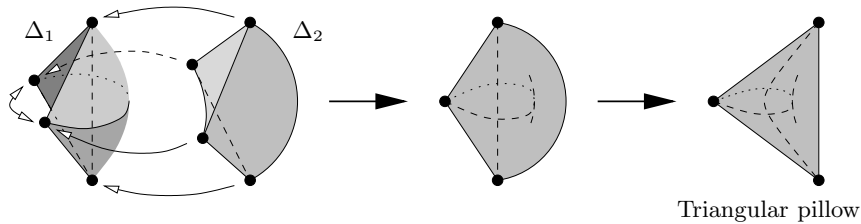}
        \caption{The two-tetrahedron triangular pillow
            at the centre of a 4-block}
        \label{fig-pillow}
    \end{figure}

    We begin by folding together two faces of $\Delta_1$, and then
    wrapping $\Delta_2$ around the remaining two faces as illustrated
    in Figure~\ref{fig-pillow}.  This forms a \emph{triangular pillow}
    with three vertices, three boundary edges, two internal edges,
    and two boundary faces.

    \begin{figure}[htb]
        \centering
        \includegraphics{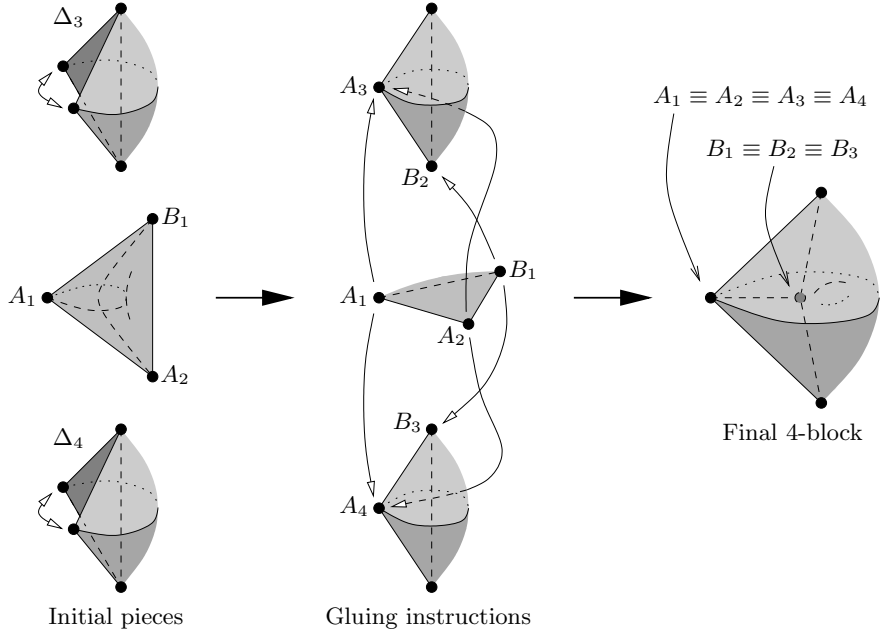}
        \caption{Building a 4-block from two tetrahedra and a
            triangular pillow}
        \label{fig-block}
    \end{figure}

    Next we fold together two faces of $\Delta_3$ and two faces of
    $\Delta_4$, as illustrated in the leftmost column of
    Figure~\ref{fig-block}.  To finish, we join the pillow to
    both $\Delta_3$ and $\Delta_4$ as illustrated in the central column
    of Figure~\ref{fig-block}---the upper face $A_1B_1A_2$ of the pillow
    is glued to the lower face $A_3B_2A_3$ of $\Delta_3$, and the
    lower face $A_1B_1A_2$ of the pillow is glued to the upper face
    $A_4B_3A_4$ of $\Delta_4$.

    The final result is shown in the rightmost column of
    Figure~\ref{fig-block}, with three boundary vertices and one
    internal vertex.  The triangular pillow is buried in the middle of
    this structure, wrapped around the internal vertex; for simplicity
    the two edges inside the pillow are not shown.
\end{defn}

\begin{defn}[Pathological triangulation $\path_k$]
    For each integer $k \geq 1$, the \emph{pathological triangulation}
    $\path_k$ is constructed from $n=4k$ tetrahedra in the following manner.

    From these $4k$ tetrahedra we build $k$ distinct 4-blocks, labelled
    $\block_1,\ldots,\block_k$.  Within each 4-block $\block_i$ we label the
    three boundary vertices $P_i,Q_i,R_i$, where $P_i$ sits between both
    boundary triangles as illustrated in Figure~\ref{fig-path}.

    \begin{figure}[htb]
        \centering
        \includegraphics{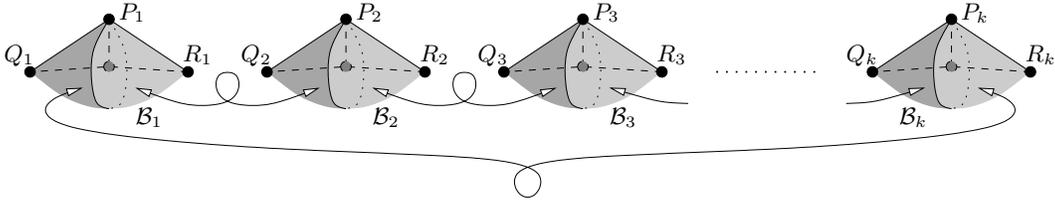}
        \caption{Building the pathological triangulation $\path_k$
            from $k$ distinct 4-blocks}
        \label{fig-path}
    \end{figure}

    For each $i=1,\ldots,k$ we join blocks $\block_i$ and $\block_{i+1}$
    as follows (where $\block_{k+1}$ is taken to mean $\block_1$).
    Triangle $P_iP_iR_i$ is joined to triangle $Q_{i+1}P_{i+1}P_{i+1}$;
    note that this is ``twisted'', not a direct gluing, since it maps
    $P_i \leftrightarrow Q_{i+1}$ and $P_{i+1} \leftrightarrow R_i$.
    There are in fact two ways this gluing can be performed (one a
    reflection of the other); we resolve this ambiguity by orienting each
    block consistently, and then choosing the gluing that preserves
    orientation.

    An effect of these gluings is to identify all of the $P_i$, $Q_i$
    and $R_i$ to a single vertex, so that
    $\path_k$ has $k+1$ vertices in total (counting also the
    $k$ internal vertices from each original block).
\end{defn}

It is not clear that each $\path_k$ is a 3-manifold triangulation (in
particular, that $\path_k$ looks like $\R^3$ in the vicinity of each
vertex).  The following sequence of results proves this by showing that
every $\path_k$ is in fact a triangulation of the 3-sphere.

\begin{lemma} \label{l-block}
    A 4-block is a triangulation of the 3-ball (i.e., the solid
    3-dimensional ball), with
    a boundary consisting of two triangles in the formation shown in
    Figure~\ref{fig-blockbdry}.
\end{lemma}

\begin{figure}[htb]
    \centering
    \includegraphics{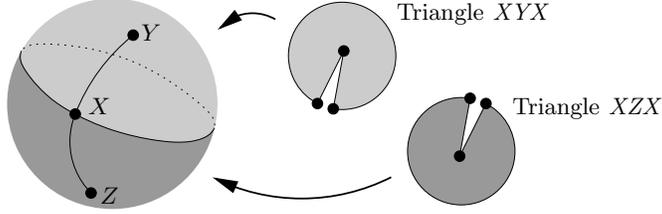}
    \caption{A 3-ball whose boundary consists of two triangles}
    \label{fig-blockbdry}
\end{figure}

\begin{proof}
    This is evident from the construction in Figure~\ref{fig-block}.
    It can also be verified computationally using the software
    package {\regina} \cite{regina}, which implements 3-sphere and
    3-ball recognition \cite{burton04-regina}.
\end{proof}

\begin{lemma} \label{l-ball-join}
    Let $\tri_1$ and $\tri_2$ each be triangulations of the 3-ball
    with boundaries in the formation shown in Figure~\ref{fig-blockbdry}.
    If we identify one boundary triangle of $\tri_1$ with one boundary
    triangle of $\tri_2$ under any of the six possible identifications,
    the result is always another triangulation of the 3-ball
    with boundary in the formation shown in Figure~\ref{fig-blockbdry}.
\end{lemma}

\begin{lemma} \label{l-ball-wrap}
    Let $\tri$ be a triangulation of the 3-ball with boundary in the
    formation shown in Figure~\ref{fig-blockbdry}.  If we
    identify the two boundary triangles under any of the three
    possible orientation-preserving identifications, the result is
    always a closed 3-manifold triangulation of the 3-sphere.
\end{lemma}

\begin{proof}
    Both of these results are essentially properties of 3-manifolds,
    not their underlying triangu\-lations---if they hold for any selection
    of triangulations $\tri_1,\tri_2,\tri$ then they must hold for all
    such selections.  We verify these results using {\regina} by
    choosing 4-blocks for our triangulations and testing all
    six/three possible identifications.
\end{proof}

Since each $\path_k$ is built by joining
together 4-blocks along boundary triangles in an orientation-preserving
fashion, the following result follows
immediately from Lemmata~\ref{l-block}--\ref{l-ball-wrap}.

\begin{corollary}
    For each $k \geq 1$, $\path_k$ is a closed 3-manifold triangulation
    of the 3-sphere.
\end{corollary}

We turn our attention now to counting the vertex normal surfaces for
each triangulation $\path_k$.  Recalling that $k=n/4$,
the following result shows that for these pathological triangulations
we have $\scount \oin \Theta(17^{n/4}) \simeq \Theta(2.03^n)$.

\begin{lemma} \label{l-worst-4k}
    For each $k \geq 1$, $\path_k$ has precisely
    $\scount = 17^k + k$ vertex normal surfaces.
\end{lemma}

\begin{proof}
    %
    Consider a single 4-block with boundary vertices labelled $P,Q,R$
    as before, and let $S$ denote the internal vertex.
    Define $\baraw$, $\bbraw$ and $\bcraw$ to be small loops on the
    4-block boundary surrounding $P$, $Q$ and $R$ respectively,
    as illustrated in Figure~\ref{fig-curves}.

    \begin{figure}[htb]
        \centering
        \includegraphics{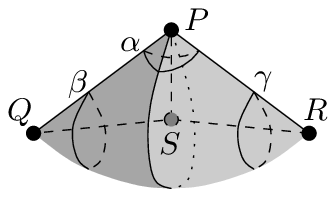}
        \caption{The curves $\baraw,\bbraw,\bcraw$ on the boundary of a
            4-block}
        \label{fig-curves}
    \end{figure}

    Using the software package {\regina}, we can construct the
    projective solution space for this 4-block.  There are
    17 admissible vertices in total, corresponding to 17 vertex
    normal surfaces: one with empty boundary, and 16 whose boundary
    consists of some combination of $\baraw$, $\bbraw$ and $\bcraw$.
    These surfaces are summarised in Table~\ref{tab-dd-block},
    and we label them $\s{a},\s{b},\ldots,\s{q}$ as shown.

    \begin{table}[htb]
    \small
    \centering
    \begin{tabular}{c|c|l}
    \textbf{Label} & \textbf{Boundary} &
        \multicolumn{1}{|c}{\textbf{Description}} \\
    \hline
    $\s{a}$ & ---     & Small sphere around internal vertex $S$ \\
    $\s{b}$ & $\ba$   & Small disc around boundary vertex $P$ \\
    $\s{c}$ & $\bb$   & Small disc around boundary vertex $Q$ \\
    $\s{d}$ & $\bc$   & Small disc around boundary vertex $R$ \\
    $\s{e}$ & $\ba$   & Tube from $P$ to $S$, closed around $S$ \\
    $\s{f}$ & $\bb$   & Tube from $Q$ to $S$, closed around $S$ \\
    $\s{g}$ & $\bc$   & Tube from $R$ to $S$, closed around $S$ \\
    $\s{h}$ & $\bab$  & Tube from $P$ to $Q$ via $S$, open at both ends \\
    $\s{i}$ & $\bac$  & Tube from $P$ to $R$ via $S$, open at both ends \\
    $\s{j}$ & $\bbc$  & Tube from $Q$ to $R$ via $S$, open at both ends \\
    $\s{k}$ & $\babc$ & Forked tube joining all of $P,Q,R$ via $S$, open
                        at all three ends \\
    $\s{l}$ & $\ba$   & Surface $\s{b}$ with large ``balloon'' disc attached
                        inside the pillow \\
    $\s{m}$ & $\ba$   & Surface $\s{b}$ with punctured torus attached
                        inside the pillow \\
    $\s{n}$ & $\ba$   & Surface $\s{e}$ with punctured torus attached
                        inside the pillow \\
    $\s{o}$ & $\bab$  & Surface $\s{h}$ with punctured torus attached
                        inside the pillow \\
    $\s{p}$ & $\bac$  & Surface $\s{i}$ with punctured torus attached
                        inside the pillow \\
    $\s{q}$ & $\babc$ & Surface $\s{k}$ with punctured torus attached
                        inside the pillow
    \end{tabular}
    \caption{The $17$ vertex normal surfaces within a 4-block}
    \label{tab-dd-block}
    \end{table}

    It is important to note that $\s{a},\s{b},\ldots,\s{q}$ are all
    \emph{compatible}; that is, no combination of their vectors can
    ever violate the quadrilateral constraints.\footnote{This is
    because, within each tetrahedron, we observe that two of the three
    quadrilateral types never appear \emph{anywhere} amongst the surfaces
    $\s{a},\s{b},\ldots,\s{q}$.}
    This is an unusual
    but extremely helpful state of affairs, since we can effectively
    ignore the quadrilateral constraints from here onwards.

    Now consider the full set of 4-blocks $\block_1,\ldots,\block_k$;
    let $\s{a}_i,\s{b}_i,\ldots,\s{q}_i$ denote the corresponding surfaces in
    $\block_i$, and let $\baraw_i,\bbraw_i,\bcraw_i$ denote the
    corresponding boundary curves.
    Any normal surface in $\path_k$ is a union of normal
    surfaces in $\block_1,\ldots,\block_k$, and hence can be expressed as
    \begin{equation*}
        (\lambda_{1,1}\, \s{a}_1 + \ldots + \lambda_{1,17}\, \s{q}_1) +
         \ldots +
        (\lambda_{k,1}\, \s{a}_k + \ldots + \lambda_{k,17}\, \s{q}_k)
    \end{equation*}
    for some family of constants
    $\lambda_{1,1},\ldots,\lambda_{k,17} \geq 0$.
    In this form, it can be shown\footnote{The argument uses the
    facts that curves $\baraw_i,\bbraw_i,\bcraw_i$ surround
    vertices $P_i,Q_i,R_i$ respectively, and that all of these vertices
    are identified together in the overall triangulation $\path_k$.}
    that the matching equations for $\path_k$
    reduce to the following statement:
    \begin{quote}
        There is some non-negative $\mu \in \R$ such that,
        for every $i$, the sum
        $\lambda_{i,1}\, \s{a}_i + \ldots + \lambda_{i,17}\, \s{q}_i$
        has boundary $\mu\baraw_i+\mu\bbraw_i+\mu\bcraw_i$.
    \end{quote}
    In other words, the portion of the overall surface within each
    4-block $\block_i$ must have boundary
    $\mu\baraw_i+\mu\bbraw_i+\mu\bcraw_i$, where $\mu$ is independent of $i$.

    Return now to a single 4-block with admissible
    vertices $\s{a},\ldots,\s{q}$, and let
    $\lambda_{1}\, \s{a} + \ldots + \lambda_{17}\, \s{q}$
    be some point in the projective solution space for this 4-block.
    We can ensure that the corresponding surface has boundary of the
    form $\mu\baraw + \mu\bbraw + \mu\bcraw$ by imposing the following
    constraints:\footnote{Each line in these constraints corresponds to
    a section of the \emph{Boundary} column in Table~\ref{tab-dd-block}.}
    \[ \small
    \begin{array}{l@{\quad}l@{}l@{}l@{}l@{}l@{}l@{}l@{}l@{}l@{}l@{}l@{}l@{}l@{}l@{}l@{}l@{}}
        & \lambda_{2}+{} & & & \lambda_{5}+{} & & & \lambda_{8}+{} &
          \lambda_{9}+{} & & \lambda_{11}+{} & \lambda_{12}+{} &
          \lambda_{13}+{} & \lambda_{14}+{} & \lambda_{15}+{} &
          \lambda_{16}+{} & \lambda_{17} \\
        = & & \lambda_{3}+{} & & & \lambda_{6}+{} & & \lambda_{8}+{} &
          & \lambda_{10}+{} & \lambda_{11}+{} & & & & \lambda_{15}+{} &
          & \lambda_{17}  \\
        = & & & \lambda_{4}+{} & & & \lambda_{7}+{} & & \lambda_{9}+{} &
          \lambda_{10}+{} & \lambda_{11}+{} & & & & & \lambda_{16}+{} &
          \lambda_{17}
    \end{array} \]

    This has the effect of intersecting the original projective solution
    space for the 4-block with two new hyperplanes.
    A standard application of the filtered double description method
    \cite{burton10-dd} shows that the resulting polytope has 18 admissible
    vertices, described by the following 18 normal surfaces:
    the original $\s{a}$ with no boundary, and 17 new surfaces%
    \footnote{These are the six surfaces
    $(\s{c}+\s{g},\ \s{d}+\s{f},\ \mathrm{or}\ \s{j}) +
     (\s{b}\ \mathrm{or}\ \s{l})$, the five surfaces
    $\s{c}+\s{d} + (\s{b},\ \s{e},\ \s{l},\ \s{m},\ \mathrm{or}\ \s{n})$,
    and the six surfaces
    $\s{c}+\s{i}$,
    $\s{c}+\s{p}$,
    $\s{d}+\s{h}$,
    $\s{d}+\s{o}$,
    $\s{k}$ and $\s{q}$.}
    all with boundary $\baraw+\bbraw+\bcraw$.
    Within each block $\block_i$, we label these 17 new surfaces
    $\s{v}_{i,1},\ldots,\s{v}_{i,17}$.

    Given the formulation of the matching equations above, it follows
    that the normal
    surfaces in $\path_k$ are described completely by the
    linear combinations
    \[ \rho_{1,1}\, \s{v}_{1,1} + \ldots + \rho_{k,17}\, \s{v}_{k,17}
       + \eta_1 \s{a}_1 + \ldots + \eta_k \s{a}_k, \]
    where each $\rho_{i,j},\eta_i \geq 0$ and where
    $\sum_j \rho_{1,j} = \sum_j \rho_{2,j} = \ldots = \sum_j \rho_{k,j}$.
    The full projective solution space for $\path_k$ therefore has
    $17^k + k$ admissible vertices, corresponding to the
    $k$ surfaces $\s{a}_1,\ldots,\s{a}_k$ and the $17^k$ combinations
    $\s{v}_{1,j_1} + \s{v}_{2,j_2} + \ldots + \s{v}_{k,j_k}$ for
    $j_1,j_2,\ldots,j_k \in \{1,\ldots,17\}$.
\end{proof}

The pathological triangulations $\path_1,\path_2,\ldots$ cover all
sizes of the form $n=4k$.  We can generalise this construction to include
$n=4k+1$, $4k+2$ and $4k+3$ by replacing one of our 4-blocks with a
single ``exceptional'' block.  The general constructions and analyses are
detailed in the full version of this paper, and the results are summarised
in the following theorem.

\begin{theorem} \label{t-worst-cases}
    For every positive $n \neq 1,2,3,5$, there exists a closed 3-manifold
    triangulation of size $n$ whose admissible vertex count is as follows:

    \begin{equation} \label{eqn-worst-cases}
    \begin{array}{ll@{\quad\Longrightarrow\quad}l}
        n = 4k   & (k \geq 1) & \scount = 17^k + k \\
        n = 4k+1 & (k \geq 2) & \scount = 581 \cdot 17^{k-2} + k + 1 \\
        n = 4k+2 & (k \geq 1) & \scount = 69 \cdot 17^{k-1} + k \\
        n = 4k+3 & (k \geq 1) & \scount = 141 \cdot 17^{k-1} + k + 2
    \end{array}
    \end{equation}
\end{theorem}

Lemma~\ref{l-worst-4k} proves this result for the first
case $n=4k$.  For an extra measure of verification,
equation~(\ref{eqn-worst-cases}) has been confirmed numerically for
all $n \leq 14$ by building the relevant triangulations and using
{\regina} to enumerate all vertex normal surfaces.

The main result of this section is the following limit on any upper
bound for $\scount$, which follows immediately from
Theorem~\ref{t-worst-cases}.  Moreover, as we discover in the following
section, there is reason to believe that this may in fact
give the tightest possible asymptotic bound.

\begin{corollary}
    Any upper bound for the admissible vertex count $\scount$
    must grow at a rate of at least $\Omega(17^{n/4}) \simeq \Omega(2.03^n)$.
\end{corollary}

%
%

\section{Practical Growth} \label{s-practice}

We turn now to a comprehensive study of the admissible vertex count
$\scount$ for real 3-manifold triangulations.  The basis of this study
is a complete census of \emph{all} closed 3-manifold triangulations
of size $n \leq 9$.  This is a significant undertaking, and such a
census has never been compiled before; the paper \cite{burton07-nor10}
details some of the sophisticated algorithms involved.

The result is a collection of $149\,676\,922$ triangulations, each
counted once up to \emph{isomorphism} (a relabelling of tetrahedra
and their vertices).  It is worth noting that within this large collection of
triangulations there is a much smaller number of distinct 3-manifolds,
as indicated by the 3-manifold census data of Martelli and Petronio
\cite{martelli01-or9} and the author \cite{burton07-nor10}.

For each of these $\sim 150$ million triangulations we enumerate all
vertex normal surfaces using the algorithms described in
\cite{burton09-convert,burton10-dd}.
The resulting admissible vertex counts $\scount$ are summarised in
Table~\ref{tab-stats}.  All computations were performed using the
software package {\regina} \cite{regina,burton04-regina}.

\begin{table}[htb]
    \centering
    \small
    \newlength{\statwidth}
    \settowidth{\statwidth}{Std dev}
    \begin{tabular}{r|r|p{\statwidth}|p{\statwidth}|p{\statwidth}|p{\statwidth}}
    \multicolumn{1}{c|}{Number of} & \multicolumn{1}{c|}{Number of} &
    \multicolumn{4}{c}{Admissible vertex count ($\scount$)} \\
    \multicolumn{1}{c|}{tetrahedra ($n$)} &
    \multicolumn{1}{c|}{triangulations} &
    \multicolumn{1}{c|}{Mean} & \multicolumn{1}{c|}{Std dev} &
    \multicolumn{1}{c|}{Min} & \multicolumn{1}{c}{Max} \\
    \hline
    1 &             4 & \hfill  2.00 & \hfill  0.71 & \hfill 1 & \hfill   3 \\
    2 &            17 & \hfill  3.94 & \hfill  1.39 & \hfill 2 & \hfill   7 \\
    3 &            81 & \hfill  5.49 & \hfill  1.97 & \hfill 2 & \hfill  11 \\
    4 &           577 & \hfill  8.80 & \hfill  3.38 & \hfill 2 & \hfill  18 \\
    5 &        5\,184 & \hfill 13.34 & \hfill  5.49 & \hfill 4 & \hfill  36 \\
    6 &       57\,753 & \hfill 20.76 & \hfill  9.21 & \hfill 4 & \hfill  70 \\
    7 &      722\,765 & \hfill 32.17 & \hfill 15.29 & \hfill 4 & \hfill 144 \\
    8 &   9\,787\,509 & \hfill 50.20 & \hfill 25.52 & \hfill 4 & \hfill 291 \\
    9 & 139\,103\,032 & \hfill 78.49 & \hfill 42.51 & \hfill 4 & \hfill 584
    \end{tabular}
    \caption{Summary of admissible vertex counts for all triangulations
        ($n \leq 9$)}
    \label{tab-stats}
\end{table}

The figures that we see are remarkably small.  For $n=9$ tetrahedra,
although Theorem~\ref{t-ubound} places the theoretical bound at
$\simeq O(29^n)$, we have just $584$ vertex normal surfaces in the worst
case.  The mean admissible vertex count for $n=9$ is much smaller again,
evaluated at just $78.49$.  The full distribution of all admissible vertex
counts for $n=9$ is shown in the left-hand graph of Figure~\ref{fig-graphs}.

\begin{figure}[htb]
    \centering
    \includegraphics[scale=0.5]{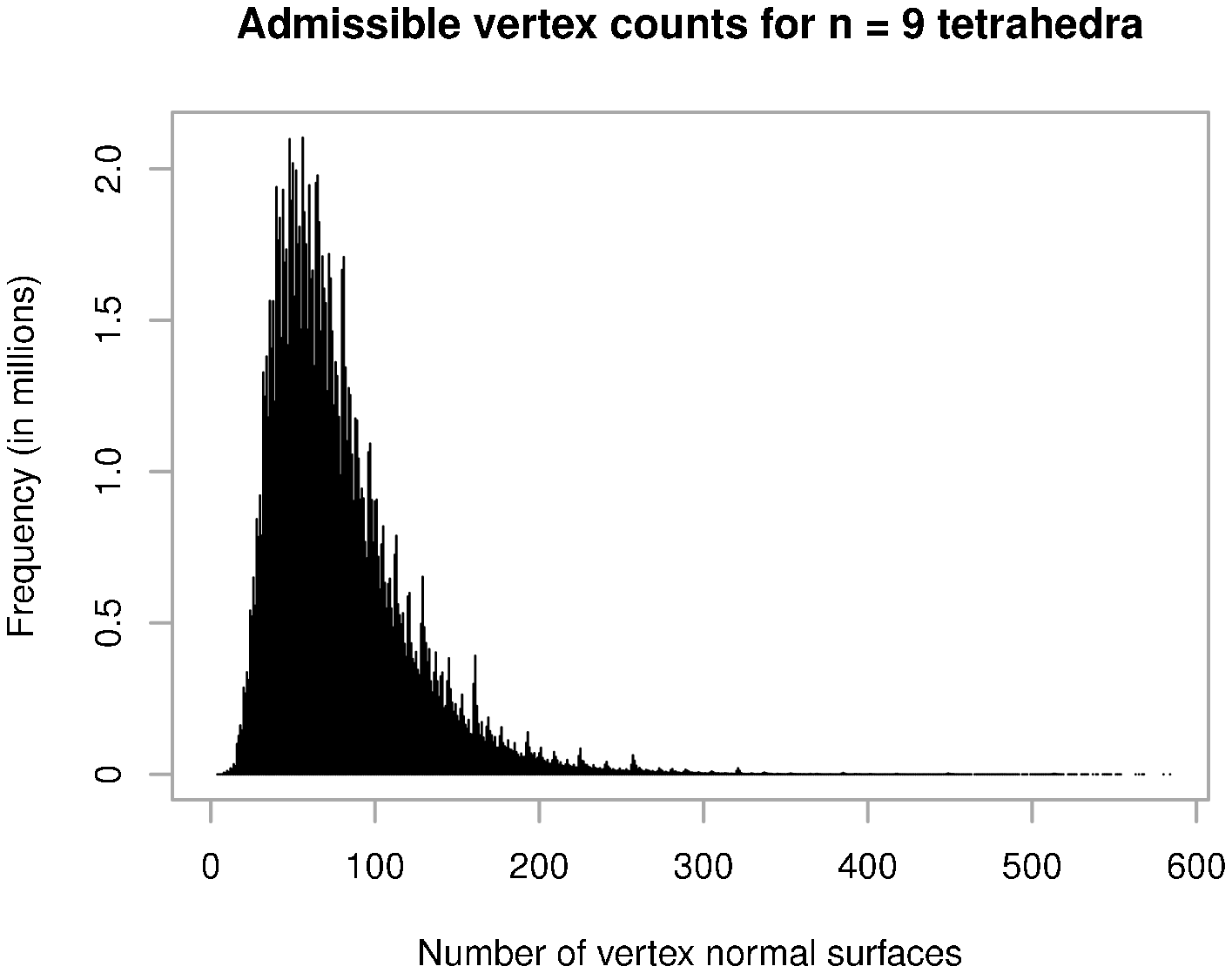}
    \qquad
    \includegraphics[scale=0.5]{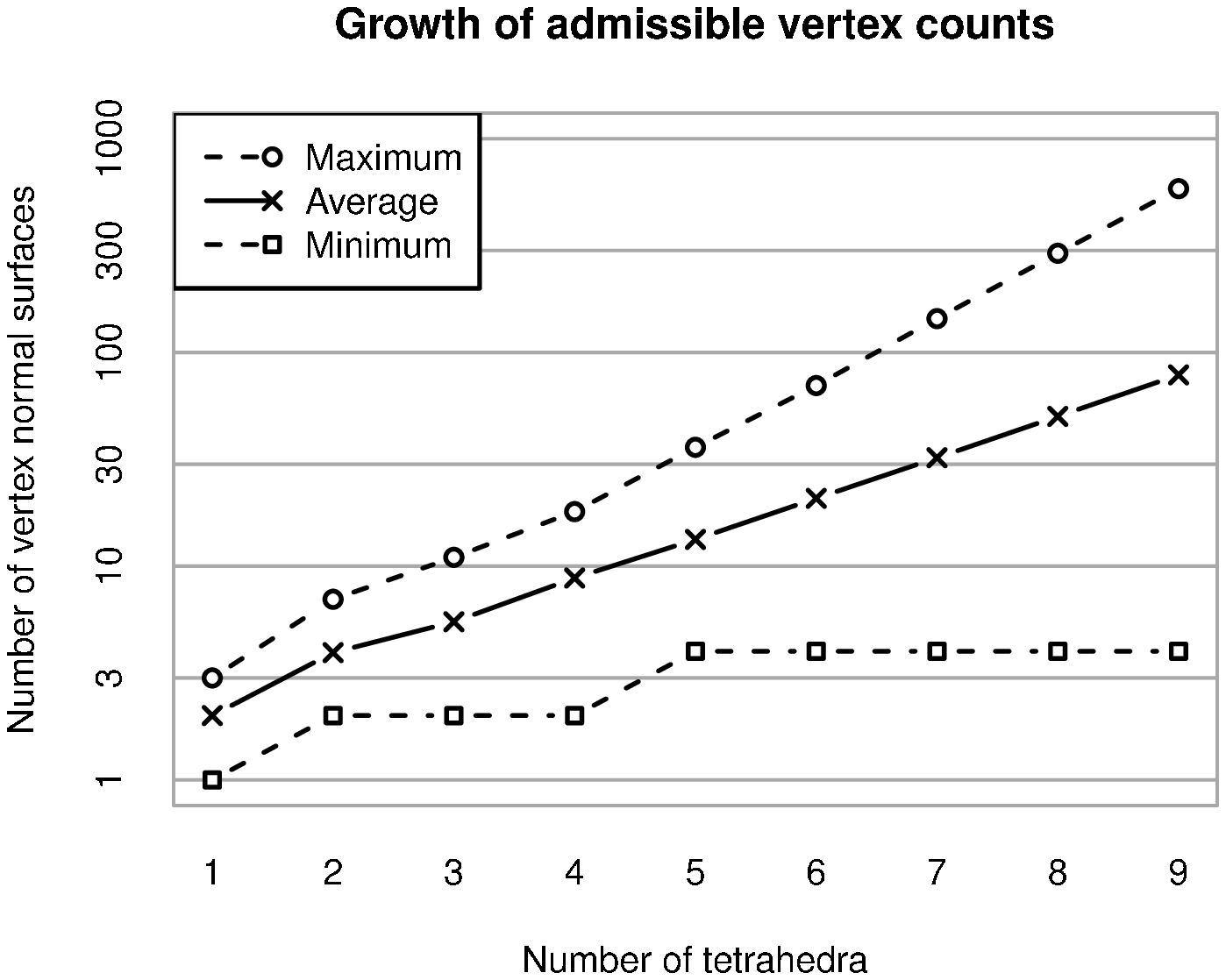}
    \caption{Aggregate results for admissible vertex counts}
    \label{fig-graphs}
\end{figure}

Indeed, our pathological triangulations $\path_1,\path_2$
are the worst cases for $n=4,8$ respectively, giving the maximum
observed values of $\scount = 17^1+1=18$ and $\scount = 17^2+2=291$.
More generally, the pathological triangulations
of Theorem~\ref{t-worst-cases} give the maximum cases in our census
wherever they are defined (i.e., $n \neq 1,2,3,5$).
This leads us to the following general conjecture:

\begin{conjecture} \label{cj-worst}
    For every positive $n \neq 1,2,3,5$, equation~(\ref{eqn-worst-cases})
    gives a tight upper bound on the admissible vertex count $\scount$.
    As a consequence, we have $\scount \oin O(17^{n/4})$.
\end{conjecture}

The growth rate of $\scount$ for $n=1,\ldots,9$ is illustrated in the
right-hand graph of
Figure~\ref{fig-graphs} (note that the vertical axis is plotted on a log
scale).  The growth rate of the maximum $\scount$ is roughly
$17^{n/4} \simeq 2.03^n$ as suggested above; the growth rate of the average
$\smeanraw$ is in the range $1.5^n$ to $1.6^n$.  This is just
below the Fibonacci growth rate of $\phi^n \simeq 1.62^n$.
Indeed, if we let $\smean{n}$ denote the mean admissible vertex count
amongst all triangulations of size $n$, we find that
$\smean{n} < \smean{n-1} + \smean{n-2}$ throughout our census.
This leads us to our next general conjecture:

\begin{conjecture} \label{cj-mean}
    For every $n \geq 3$, the mean admissible vertex count $\smean{n}$
    satisfies the relation $\smean{n} < \smean{n-1} + \smean{n-2}$.
    As a consequence, $\smean{n}$ is bounded above by $O(\phi^n)$
    where $\phi = (1+\sqrt{5})/2$.
\end{conjecture}

In particular, our census analysis gives us the following computational result:

\begin{theorem}
    Conjectures~\ref{cj-worst} and \ref{cj-mean} are true for $n \leq 9$.
\end{theorem}

%
%

\section{Conclusions}

We have pushed the theoretical bounds on the admissible
vertex count $\scount$ from both directions, and we have shown through an
exhaustive study of $\sim 150$ million triangulations that
$\scount$ is surprisingly small in practice.  We close with a brief
discussion of the implications of this study.

Most importantly, it suggests that topological algorithms that employ
normal surfaces might not be as infeasible as theory suggests.
Hints of this have already been seen with the quadrilateral-to-standard
conversion algorithm for normal surfaces \cite{burton09-convert}, which
(against theoretical expectations) appears to have a running time
polynomial in its output size.

In many fields, a census for size $n \leq 9$ might not seem large
enough for drawing conclusions and conjectures.  However, there is
evidence elsewhere to suggest that 3-manifold triangulations are
flexible enough for important patterns to establish themselves for
very low $n$.  For example, the papers \cite{burton07-nor7,matveev98-or6}
discuss several combinatorial patterns for $n \leq 6$; these patterns
have later been found to generalise well for larger $n$
\cite{burton07-nor10,martelli01-or9}, and some are now
proven in general \cite{jaco09-minimal-lens,jaco09-coverings}.

Finally, it is clear from this practical study that the theoretical
bounds on $\scount$ still have much room for improvement.  One
possible direction is to incorporate the quadrilateral constraints
directly into McMullen's theorem.  This is difficult because the
quadrilateral constraints break convexity, but the outcome may be
significantly closer to the $O(17^{n/4})$ that we see in practice.

%
%

\section*{Acknowledgements}

The author is grateful to both the University of Victoria (Canada)
and the Victorian Partnership for Advanced Computing (Australia)
for the use of their excellent computing resources,
and to the Australian Research Council for their support
under the Discovery Projects funding scheme (project DP1094516).

%
%

\small
\bibliographystyle{amsplain}
\bibliography{pure}

\newcommand{\noopsort}[1]{}
\providecommand{\bysame}{\leavevmode\hbox to3em{\hrulefill}\thinspace}
\providecommand{\MR}{\relax\ifhmode\unskip\space\fi MR }
\providecommand{\MRhref}[2]{%
  \href{http://www.ams.org/mathscinet-getitem?mr=#1}{#2}
}
\providecommand{\href}[2]{#2}
\begin{thebibliography}{10}

\bibitem{agol02-knotgenus}
Ian Agol, Joel Hass, and William Thurston, \emph{3-manifold knot genus is
  {NP}-complete}, STOC '02: Proceedings of the Thiry-Fourth Annual {ACM}
  Symposium on Theory of Computing, ACM Press, 2002, pp.~761--766.

\bibitem{avis97-howgood-compgeom}
David Avis, David Bremner, and Raimund Seidel, \emph{How good are convex hull
  algorithms?}, Comput. Geom. \textbf{7} (1997), no.~5-6, 265--301.

\bibitem{regina}
Benjamin~A. Burton, \emph{Regina: Normal surface and 3-manifold topology
  software}, \texttt{http://\allowbreak regina.\allowbreak
  sourceforge.\allowbreak net/}, 1999--2009.

\bibitem{burton04-regina}
\bysame, \emph{Introducing {R}egina, the 3-manifold topology software},
  Experiment. Math. \textbf{13} (2004), no.~3, 267--272.

\bibitem{burton07-nor10}
\bysame, \emph{Enumeration of non-orientable 3-manifolds using face-pairing
  graphs and union-find}, Discrete Comput. Geom. \textbf{38} (2007), no.~3,
  527--571.

\bibitem{burton07-nor7}
\bysame, \emph{Structures of small closed non-orientable 3-manifold
  triangulations}, J. Knot Theory Ramifications \textbf{16} (2007), no.~5,
  545--574.

\bibitem{burton09-convert}
\bysame, \emph{Converting between quadrilateral and standard solution sets in
  normal surface theory}, Algebr. Geom. Topol. \textbf{9} (2009), no.~4,
  2121--2174.

\bibitem{burton09-quadoct}
\bysame, \emph{Quadrilateral-octagon coordinates for almost normal surfaces},
  To appear in Experiment. Math., \texttt{arXiv:\allowbreak 0904.3041}, April
  2009.

\bibitem{burton10-dd}
\bysame, \emph{Optimizing the double description method for normal surface
  enumeration}, Math. Comp. \textbf{79} (2010), no.~269, 453--484.

\bibitem{burton09-ws}
Benjamin~A. Burton, J.~Hyam Rubinstein, and Stephan Tillmann, \emph{The
  {W}eber-{S}eifert dodecahedral space is non-{H}aken}, Preprint,
  \texttt{arXiv:\allowbreak 0909.4625}, September 2009.

\bibitem{callahan99-cuspedcensus}
Patrick~J. Callahan, Martin~V. Hildebrand, and Jeffrey~R. Weeks, \emph{A census
  of cusped hyperbolic 3-manifolds}, Math. Comp. \textbf{68} (1999), no.~225,
  321--332.

\bibitem{dyer83-complexity}
M.~E. Dyer, \emph{The complexity of vertex enumeration methods}, Math. Oper.
  Res. \textbf{8} (1983), no.~3, 381--402.

\bibitem{grunbaum03}
Branko Gr{\"u}nbaum, \emph{Convex polytopes}, 2nd ed., Graduate Texts in
  Mathematics, no. 221, Springer, New York, 2003.

\bibitem{haken61-knot}
Wolfgang Haken, \emph{Theorie der {N}ormalfl{\"a}chen}, Acta Math. \textbf{105}
  (1961), 245--375.

\bibitem{haken62-homeomorphism}
\bysame, \emph{{\"U}ber das {H}om{\"o}omorphieproblem der
  3-{M}annigfaltigkeiten. {I}}, Math. Z. \textbf{80} (1962), 89--120.

\bibitem{hass99-knotnp}
Joel Hass, Jeffrey~C. Lagarias, and Nicholas Pippenger, \emph{The computational
  complexity of knot and link problems}, J. Assoc. Comput. Mach. \textbf{46}
  (1999), no.~2, 185--211.

\bibitem{jaco84-haken}
William Jaco and Ulrich Oertel, \emph{An algorithm to decide if a
  {$3$}-manifold is a {H}aken manifold}, Topology \textbf{23} (1984), no.~2,
  195--209.

\bibitem{jaco09-minimal-lens}
William Jaco, Hyam Rubinstein, and Stephan Tillmann, \emph{Minimal
  triangulations for an infinite family of lens spaces}, J. Topol. \textbf{2}
  (2009), no.~1, 157--180.

\bibitem{jaco03-0-efficiency}
William Jaco and J.~Hyam Rubinstein, \emph{0-efficient triangulations of
  3-manifolds}, J. Differential Geom. \textbf{65} (2003), no.~1, 61--168.

\bibitem{jaco09-coverings}
William Jaco, J.~Hyam Rubinstein, and Stephan Tillmann, \emph{Coverings and
  minimal triangulations of 3-manifolds}, To appear in Algebr. Geom. Topol.,
  \texttt{arXiv:\allowbreak 0903.0112}, February 2009.

\bibitem{jaco95-algorithms-decomposition}
William Jaco and Jeffrey~L. Tollefson, \emph{Algorithms for the complete
  decomposition of a closed {$3$}-manifold}, Illinois J. Math. \textbf{39}
  (1995), no.~3, 358--406.

\bibitem{khachiyan08-hard}
Leonid Khachiyan, Endre Boros, Konrad Borys, Khaled Elbassioni, and Vladimir
  Gurvich, \emph{Generating all vertices of a polyhedron is hard}, Discrete
  Comput. Geom. \textbf{39} (2008), no.~1-3, 174--190.

\bibitem{kneser29-normal}
Hellmuth Kneser, \emph{Geschlossene {F}l{\"a}chen in dreidimensionalen
  {M}annigfaltigkeiten}, Jahresbericht der Deut. Math. Verein. \textbf{38}
  (1929), 248--260.

\bibitem{li10-genus}
Tao Li, \emph{An algorithm to determine the {H}eegaard genus of a 3-manifold},
  Preprint, \texttt{arXiv:\allowbreak 1002.1958}, February 2010.

\bibitem{markov60-insolubility}
A.~A. Markov, \emph{Insolubility of the problem of homeomorphy}, Proc.
  Internat. Congress Math. 1958, Cambridge Univ. Press, New York, 1960,
  pp.~300--306.

\bibitem{martelli01-or9}
Bruno Martelli and Carlo Petronio, \emph{Three-manifolds having complexity at
  most 9}, Experiment. Math. \textbf{10} (2001), no.~2, 207--236.

\bibitem{massey91}
William~S. Massey, \emph{A basic course in algebraic topology}, Graduate Texts
  in Mathematics, no. 127, Springer-Verlag, New York, 1991.

\bibitem{matsumoto00-fig8}
Saburo Matsumoto and Richard Rannard, \emph{The regular projective solution
  space of the figure-eight knot complement}, Experiment. Math. \textbf{9}
  (2000), no.~2, 221--234.

\bibitem{matveev98-or6}
Sergei~V. Matveev, \emph{Tables of 3-manifolds up to complexity 6},
  Max-Planck-Institut f{\"u}r Mathematik Preprint Series (1998), no.~67,
  available from \texttt{http://www.\allowbreak mpim-bonn.\allowbreak
  mpg.\allowbreak de/\allowbreak html/\allowbreak pre\-prints/\allowbreak
  preprints.html}.

\bibitem{matveev05-or11}
\bysame, \emph{Recognition and tabulation of three-dimensional manifolds},
  Dokl. Akad. Nauk \textbf{400} (2005), no.~1, 26--28.

\bibitem{mcmullen70-ubt}
P.~McMullen, \emph{The maximum numbers of faces of a convex polytope},
  Mathematika \textbf{17} (1970), 179--184.

\bibitem{rubinstein95-3sphere}
J.~Hyam Rubinstein, \emph{An algorithm to recognize the {$3$}-sphere},
  Proceedings of the International Congress of Mathematicians ({Z}{\"u}rich,
  1994), vol.~1, Birkh{\"a}user, 1995, pp.~601--611.

\bibitem{rubinstein97-3sphere}
\bysame, \emph{Polyhedral minimal surfaces, {H}eegaard splittings and decision
  problems for {$3$}-di\-men\-sion\-al manifolds}, Geometric Topology
  ({A}thens, GA, 1993), AMS/IP Stud. Adv. Math., vol.~2, Amer. Math. Soc.,
  1997, pp.~1--20.

\bibitem{thompson94-thinposition}
Abigail Thompson, \emph{Thin position and the recognition problem for {$S\sp
  3$}}, Math. Res. Lett. \textbf{1} (1994), no.~5, 613--630.

\bibitem{thurston78-lectures}
William~P. Thurston, \emph{The geometry and topology of 3-manifolds}, Lecture
  notes, Princeton University, 1978.

\bibitem{thurston82-geometrisation}
\bysame, \emph{Three-dimensional manifolds, {K}leinian groups and hyperbolic
  geometry}, Bull. Amer. Math. Soc. (N.S.) \textbf{6} (1982), no.~3, 357--381.

\bibitem{tollefson98-quadspace}
Jeffrey~L. Tollefson, \emph{Normal surface {$Q$}-theory}, Pacific J. Math.
  \textbf{183} (1998), no.~2, 359--374.

\end{thebibliography}

%
%

\bigskip
\smallskip
\noindent
Benjamin A.~Burton \\
School of Mathematics and Physics, The University of Queensland \\
Brisbane QLD 4072, Australia \\
(bab@maths.uq.edu.au)

\end{document}